\documentclass[12pt,twoside]{article}
\usepackage{a4}
\usepackage{jgg}
\usepackage{graphics}
\usepackage{epstopdf}
\usepackage{overpic}
\usepackage{amsfonts,amssymb,amsmath,amsthm}

\usepackage{dsfont}
\usepackage{color}
\usepackage{nicefrac}

\newtheoremstyle{mystyle}
  {3pt}
  {3pt}
  {}
  {}
  {\bfseries}
  {:}
  {0.5 em}
  {}

\theoremstyle{mystyle}
\newtheorem{definition}{Definition}

\newtheorem{remark}{Remark}
\newtheorem{theorem}{Theorem}
\newtheorem{example}{Example}

\usepackage[latin1]{inputenc}	
\usepackage{placeins} 

\usepackage{hyperref}

\newcommand{\clifford}{\mathcal{C}\ell}

\begin{document}
\begin{JGGarticle}
{Reflections in Conics, Quadrics and Hyperquadrics via Clifford Algebra}
{}
{Daniel Klawitter}
{\JGGaddress{
Dresden University of Technology, Germany}
}
\begin{JGGabstract}\\
In this article we present a new and not fully employed geometric algebra model. With this model a generalization of the conformal geometric algebra model is achieved. We discuss the geometric objects that can be represented. Furthermore, we show that the Pin group of this geometric algebra corresponds to the group of inversions with respect to quadrics in principal position. We discuss the construction for the two- and three-dimensional case in detail and give the construction for arbitrary dimension.\\
[1mm]{\em Key Words: Clifford algebra, geometric algebra, generalized inversion, conic, quadric, hyperquadric}.\\
{\em MSC2010: 15A66, 51B99, 51M15, 51N15}.
\end{JGGabstract}

\section{Algebraic Background}
\noindent Before we introduce the quadric geometric algebra we review some algebraic background.
\subsection{Geometric Algebra}
\begin{definition}
Let $V$ be a real valued vector space of dimension $n$. Furthermore, let $b:V\mapsto \mathds{R}$ be a quadratic form on $V$. The pair $(V,b)$ is called \emph{quadratic space}.
\end{definition}
\noindent
We denote the Matrix belonging to $b$ by $\mathrm{B}_{ij}$ with $1\leq i,j\leq n$. Therefore $b(x_i,x_j)=\mathrm{B}_{ij}$ for some basis vectors $x_i$ and $x_j$.
\begin{definition}
The Clifford algebra is defined by the relations
\begin{equation}
\label{EQ1}
x_ix_j+x_jx_i=2\mathrm{B}_{ij},\quad 1\leq i,j\leq n.
\end{equation}
\end{definition}
\noindent
Usually the algebra is denoted by $\clifford(V,b)$.
By Silvester's law of inertia we can always find a basis $\lbrace e_1,\dots,e_n \rbrace$ of $V$ such that $e_i^2$ is either $1,-1$ or $0$.
\begin{definition}
The number of basis vectors that square to $(1,-1,0)$ is called \emph{signature} $(p,q,r)$. If $r\neq 0$ we call the geometric algebra \emph{degenerate}. We will denote this Clifford algebra by $\clifford_{(p,q,r)}$.
\end{definition}
\begin{remark}
A quadratic real space with signature $(p,q,0)$ is abbreviated by $\mathds{R}^{p,q}$.
\end{remark}
\noindent
With the new basis $\lbrace e_1,\dots,e_n \rbrace$ the relations \eqref{EQ1} become
\begin{equation*}
e_ie_j+e_je_i=0,\quad i\neq j \mbox{ and } e_ie_i=\mathrm{B}_{ii}.
\end{equation*}
In the remainder of this paper we shall abbreviate the product of basis elements with lists
\begin{equation*}
e_{12\dots k}:=e_1e_2\dots e_k,\, \mbox{with $0\leq k\leq n$}.
\end{equation*}
The $2^n$ monomials 
\begin{equation*}
e_{i_1}e_{i_2}\dots e_{i_k},\quad 0\leq k\leq n
\end{equation*}
form the standard basis of the Clifford algebra. The Clifford algebra and the exterior algebra are canonically isomorphic (as vector spaces). The dimension of a Clifford algebra is calculated by
\begin{equation*}
\dim \clifford_{(p,q,r)}=\sum\limits_{i=0}^n{\dim {\bigwedge}^i V}=\sum\limits_{i=0}^n\binom{n}{i}=2^n.
\end{equation*}
Moreover, the Clifford algebra $\clifford_{(p,q,r)}$ possesses a $\mathds{Z}_2$-grading,\ {\it i.e.}, it can be decomposed in an even and an odd part
\begin{equation*}
\clifford_{(p,q,r)}=\clifford_{(p,q,r)}^+ \oplus \clifford_{(p,q,r)}^-=\bigoplus\limits_{\substack{i=0\\i \text{ even}}}^n{{\bigwedge}^i V} \oplus \bigoplus\limits_{\substack{i=0\\i \text{ odd}}}^n{{\bigwedge}^i V}.
\end{equation*}
The even part $\clifford_{(p,q,r)}^+$ is a subalgebra, because the product of two even graded monomials must be even graded since the generators cancel only in pairs. Elements contained in ${\bigwedge}^i V$ are called \emph{$i$-blades} and any $\mathds{R}$-linear combination of $i$-blades is called a \emph{multi-vector}. The product of invertible $1$-blades is called a \emph{versor}.
\subsection{Clifford Algebra Involutions}
\noindent For our purposes two involutions that exist on each Clifford algebra are interesting. The \emph{conjugation} is an \emph{anti-automorphism} denoted by an asterisk, see \cite{selig:geometricfundamentalsofrobotics}. Its effect on generators is given by $e_i^\ast=-e_i$. There is no effect on scalars. Extending the conjugation by using linearity yields
\begin{equation}
(e_{i_1}e_{i_2}\dots e_{i_k})^\ast=(-1)^k e_{i_k}\dots e_{i_2}e_{i_1},\quad 0\leq k \leq n,\ 1 \leq i_1<\ldots <i_k\leq n.
\end{equation}
The geometric product of a $1$-blade $\mathfrak{v}=\sum\limits_{i=1}^{n}{x_ie_i}\in{\bigwedge}^1 V$ with its conjugate results in
\begin{equation*}
\mathfrak{v}\mathfrak{v}^\ast=-x_1^2-x_2^2-\dots-x_p^2+x_{p+1}^2+\dots +x_{p+q}^2=-b(\mathfrak{v},\mathfrak{v}).
\end{equation*}
The map $N:\clifford_{(p,q,r)}\mapsto\clifford_{(p,q,r)}$ with $N(\mathfrak{v}):=\mathfrak{v}\mathfrak{v}^\ast$ is called the \emph{norm} of the Clifford algebra and the inverse of a vector $\mathfrak{v}\in\bigwedge^1 V$ is computed by $\mathfrak{v}^{-1}:=\frac{\mathfrak{v}^\ast}{N(\mathfrak{v})}$. For general multi-vectors the determination can be found in \cite{fontijne:EfficientImplementationofGeometricAlgebra}. Note that in general not every element is invertible.
The other involution we are dealing with is the \emph{main involution}. It is denoted by $\alpha$ and defined by
\begin{equation*}
\alpha(e_{i_1}e_{i_2}\dots e_{i_k})=(-1)^k e_{i_1}e_{i_2}\dots e_{i_k},\quad 0\leq k \leq n,\ 1 \leq i_1<\ldots <i_k\leq n.
\end{equation*}
The main involution has no effect on the even subalgebra and it commutes with the conjugation, {\it i.e.}, $\alpha(\mathfrak{M}^\ast)=\alpha(\mathfrak{\mathfrak{M}})^\ast$ for arbitrary $\mathfrak{\mathfrak{M}}\in\clifford_{(p,q,r)}$.
\subsection{Clifford Algebra Products}
\noindent On $1$-blades, {\it i.e.}, vectors $\mathfrak{a},\mathfrak{b}\in{\bigwedge}^1 V$ we can write the inner product in terms of the geometric product
\begin{equation}\label{EQ10}
\mathfrak{a}\cdot \mathfrak{b}:=\frac{1}{2}(\mathfrak{a}\mathfrak{b}+\mathfrak{b}\mathfrak{a}).
\end{equation}
A generalization of the inner product to blades can be found in \cite{hestenes:cliffordalgebra}.
For $\mathfrak{A}\in{\bigwedge}^k V,\mathfrak{B}\in{\bigwedge}^l V$ the generalized inner product is defined by
\begin{equation*}
\mathfrak{A}\cdot \mathfrak{B}:=\left[ \mathfrak{A}\mathfrak{B} \right]_{\vert k-l\vert},
\end{equation*}
where $\left[ \cdot \right]_{m}\,m\in\mathds{N}$ denotes the grade-$m$ part. There is another product on $1$-blades, {\it i.e.}, the outer (or exterior) product
\begin{equation}\label{EQ12}
\mathfrak{a}\wedge \mathfrak{b}:=\frac{1}{2}(\mathfrak{a}\mathfrak{b}-\mathfrak{b}\mathfrak{a}).
\end{equation}
This product can also be generalized to blades, see again \cite{hestenes:cliffordalgebra}. For $\mathfrak{A}\in{\bigwedge}^k V,\mathfrak{B}\in{\bigwedge}^l V$ the generalized outer product is defined by
\begin{equation*}
\mathfrak{A}\wedge \mathfrak{B}:=\left[ \mathfrak{A}\mathfrak{B} \right]_{\vert k+l\vert}.
\end{equation*}
Thus, the exterior product of $\bigwedge V$ can be expressed in terms of the geometric product.
From equation \eqref{EQ10} and \eqref{EQ12} it follows that for $1$-blades the geometric product can be written as the sum of the inner and the outer product
\begin{equation*}
\mathfrak{ab}=\mathfrak{a}\cdot \mathfrak{b}+\mathfrak{a}\wedge \mathfrak{b}.
\end{equation*}
More generally, this can be defined for multivectors with the commutator and the anti-commutator product, see \cite{perwass:geometricalgebra}. For treating geometric entities within this algebra context the definition of the \emph{inner product null space} and its dual the \emph{outer product null space} is needed.
\begin{definition}
The \emph{inner product null space} (IPNS) of a blade $\mathfrak{A}\in{\bigwedge}^k V$, cf. \cite{perwass:geometricalgebra}, is defined by
\begin{equation*}
\mathds{NI}(\mathfrak{A}):=\left\lbrace \mathfrak{v}\in{\bigwedge}^1 V:\mathfrak{v}\cdot \mathfrak{A}=0 \right\rbrace.
\end{equation*}
Moreover, the \emph{outer product null space} (OPNS) of a blade $\mathfrak{A}\in{\bigwedge}^k V$ is defined by
\begin{equation*}
\mathds{NO}(\mathfrak{A}):=\left\lbrace \mathfrak{v}\in{\bigwedge}^1 V:\mathfrak{v}\wedge \mathfrak{A}=0 \right\rbrace.
\end{equation*}
\end{definition}
\subsection{Pin and Spin groups}
\noindent With respect to the geometric product the units of a Clifford algebra denoted by $\clifford_{(p,q,r)}^\times$ form a group.
\begin{definition}
The \emph{Clifford group} is defined by
\begin{equation*}
\Gamma(\clifford_{(p,q,r)}):=\left\lbrace \mathfrak{g}\in\clifford_{(p,q,r)}^\times\mid \alpha(\mathfrak{g})\mathfrak{v} \mathfrak{g}^{-1}\in {\bigwedge}^1 V \mbox{ for all } \mathfrak{v}\in{\bigwedge}^1 V \right\rbrace.
\end{equation*}
\end{definition}
\noindent
A proof that $\Gamma(\clifford_{(p,q,r)})$ is indeed a group with respect to the geometric product can be found in \cite{gallier:cliffordalgebrascliffordgroups}.
We define two important subgroups of the Clifford group.
\begin{definition}
The \emph{Pin group} is the subgroup of the Clifford group with\linebreak $N(\mathfrak{g})=1$.
\begin{equation*}
\mbox{Pin}_{(p,q,r)}\!:\!=\left\lbrace \mathfrak{g}\in\clifford_{(p,q,r)}\mid \mathfrak{gg}^\ast\!=\!\pm 1\mbox{ and } \alpha(\mathfrak{g})\mathfrak{v}\mathfrak{g}^\ast\in {\bigwedge}^1 V \mbox{ for all } \mathfrak{v}\in{\bigwedge}^1 V \right\rbrace.
\end{equation*}
Furthermore, we define the \emph{Spin group} by $\mathrm{Pin}_{(p,q,r)}\cap \clifford_{(p,q,r)}^+$
\begin{equation*}
\mbox{Spin}_{(p,q,r)}\!:=\!\left\lbrace \mathfrak{g}\in\clifford_{(p,q,r)}^+\mid \mathfrak{gg}^\ast\!=\!\pm 1\mbox{ and } \alpha(\mathfrak{g})\mathfrak{v}\mathfrak{g}^\ast\in {\bigwedge}^1 V \mbox{ for all } \mathfrak{v}\in{\bigwedge}^1 V \right\rbrace.
\end{equation*}
\end{definition}
\begin{remark}
For non-degenerate Clifford algebras the Pin group is a double cover of the orthogonal group of the quadratic space $(V,b)$. Moreover, the Spin group is a double cover of the special orthogonal group of $(V,b)$.
\end{remark}



\section{Quadric Geometric Algebra}\label{section3}
\noindent The geometric algebra we will study was first introduced by Zamora \cite{zamora:geometricalgebra}. We discuss the planar, {\it i.e.}, two-dimensional case in detail before we move on to higher dimensions.
\subsection{The Embedding}
The quadric geometric algebra for the two-dimensional case is constructed with a Clifford algebra over a six-dimensional vector space $V=\mathds{R}^6$. In fact the term quadric could be replaced by the term conic for the two-dimensional case. Without loss of generality we call it quadric geometric algebra and abbreviate this term with Q$n$GA, where $n$ denotes the dimension of the base space. The quadratic form we are using is derived by the quadratic form of the conformal geometric algebra used in \cite{dorst:gaviewer}:
\[\mathrm{Q}=\left( \begin{array}{cccccc}
 0 & 0 & -1 & 0 & 0 & 0\\
 0 & 1 &  0 & 0 & 0 & 0\\
-1 & 0 &  0 & 0 & 0 & 0\\
 0 & 0 &  0 & 0 & 0 & -1\\
 0 & 0 &  0 & 0 & 1 & 0\\
 0 & 0 &  0 & -1 & 0 & 0\\
\end{array}\right) .\]
The signature of the resulting algebra is $(p,q,r)=(4,2,0)$.
For every axis, {\it i.e.}, the $x$- and the $y$-axes a conformal embedding is performed, see \cite{dorst:geometricalgebra}. Therefore, we have the embedding
$\eta:\mathds{R}^2\rightarrow\bigwedge^1 V$,
\begin{align}
\mathds{R}^2\ni P&\mapsto \mathfrak{p}\in{\bigwedge}^1 V,\nonumber\\
P=(x,y)^\mathrm{T}&\mapsto e_1 + x e_2 + \frac{1}{2} x^2 e_3 + e_4 + y e_5 +\frac{1}{2} y^2 e_6=\mathfrak{p}.\label{EQ21}
\end{align}
Affine points $(x,y)^\mathrm{T}\in\mathds{R}^2$ are embedded as null vectors. This means 
\begin{equation}\label{EQ22}
\eta(P)^2=0 \mbox{ for $P\in\mathds{R}^2$.}
\end{equation}
The projection on the generator subspace spanned by $e_1,e_2$, and $e_3$ is denoted by subscript $x$ and the projection on $e_4,e_5,e_6$ by subscript $y$.
Due to the fact, that the embedding is conformal (see \cite{dorst:geometricalgebra,zamora:geometricalgebra}) for both axes we get the additional conditions:
\begin{equation}\label{EQ23}
\eta(P)_x^2=0,\quad \eta (P)_y^2=0.
\end{equation}
In the following we call grade-1 elements satisfying \eqref{EQ22} and \eqref{EQ23} \emph{embedded points}. Let $P_1=(x_1,y_1)^\mathrm{T}\in\mathds{R}^2$ and $P_2=(x_2,y_2)^\mathrm{T}\in\mathds{R}^2$ be two points. The inner product of their images under $\eta$ results in
\begin{align*}
\eta(P_1)\cdot\eta(P_2)&=-\frac{1}{2}x_2^2+x_1x_2-\frac{1}{2}x_1^2-\frac{1}{2}y_2^2+y_1y_2-\frac{1}{2}y_1^2\\
&=-\frac{1}{2}\left( (x_2-x_1)^2+(y_2-y_1)^2\right)\\
&=-\frac{1}{2}d^2_E(P_1,P_2),
\end{align*}
where $d_E(P_1,P_2)$ denotes the Euclidean distance between the points $P_1$ and $P_2$. Note that this formula only is true for normalized null vectors. This means that the homogeneous factors has to be equal to one. These vectors can be interpreted as images of points under $\eta$ and we call them normalized. They are characterized by
\begin{equation}\label{EQ27}
-e_3\cdot \mathfrak{p} = 1,\quad -e_6\cdot \mathfrak{p}=1.
\end{equation}
The inner product of an arbitrary embedded point with $e_3$ or $e_6$ is constant. Therefore, we can interpret these elements as points at infinity, see \cite{dorst:geometricalgebra}. Furthermore, the combination of the conditions \eqref{EQ27} results in
\begin{equation*}
-(e_3+e_6)\cdot \mathfrak{p} =2.
\end{equation*}
Thus, the geometric entity corresponding $e_3+e_6$ can be interpreted as point at infinity. The elements $e_3$ and $e_6$ represent the ideal points corresponding to each axis and $e_3+e_6$ represents a point at infinity contained in both axes. Geometrically, these three algebra elements describe the same point, {\it i.e.}, the point at infinity $\infty$ although they differ algebraically.\\

\noindent There are grade-1 elements that satisfy the conditions \eqref{EQ22} and \eqref{EQ23} without having a preimage in $\mathds{R}^2$. For example $e_3,e_6,e_3+e_6$ and algebra elements of the form:
\begin{equation*}
\mathfrak{u}_1=e_1+x_0e_2+\frac{1}{2}x_0^2e_3+e_6,\quad\quad
\mathfrak{u}_2=e_3+e_4+y_0 e_5+\frac{1}{2}y_0^2 e_6.
\end{equation*}
If we determine the Euclidean distance of an embedded point to $\mathfrak{u}_1$ or $\mathfrak{u}_2$, the result  is a complex number and depends on $x_0$ respectively $y_0$. Hence, these elements do not represent points.

\subsection{Geometric entities}
\noindent To calculate the preimage $\eta^{-1}$ of $\mathfrak{p}\in{\bigwedge}^1V$ representing an embedded point, {\it i.e.}, an algebra element fulfilling \eqref{EQ22} and \eqref{EQ23}, we determine its IPNS ($\mathds{NI}(\mathfrak{p})$) with respect to the embedding. This is called \emph{geometric inner product null space} and dual \emph{geometric outer product null space}, see \cite{perwass:geometricalgebra}.
\begin{definition}
The \emph{geometric inner product null space} (GIPNS) and dual the \emph{geometric outer product null space} (GOPNS) of a $k$-blade $\mathfrak{A}\in{\bigwedge}^kV$ is defined by
\begin{align*}
\mathds{NI}_G(\mathfrak{A})&:=\left\lbrace (x,y)^\mathrm{T}\in\mathds{R}^2:\epsilon(x,y) \cdot \mathfrak{A}=0\right\rbrace, \\
\mathds{NO}_G(\mathfrak{A})&:=\left\lbrace (x,y)^\mathrm{T}\in\mathds{R}^2:\epsilon(x,y) \wedge \mathfrak{A}=0\right\rbrace .
\end{align*}
\end{definition}

\begin{remark}
When dealing with an algebra element and the corresponding geometric entity, we will explicitly mention what null space is meant. For example we will talk about inner product conics. This means the inner product null space defines a conic in $\mathds{R}^2$.
\end{remark}
\noindent
Before we start the examination of geometric objects occurring in this model we define special $5$-blades that are necessary to change from inner product to outer product null spaces and vice versa.
\begin{definition}\label{DEF8}
On the one hand the $5$-blade
\[\mathfrak{I}=e_2 \wedge e_5 \wedge e_1 \wedge e_4 \wedge (e_3+e_6)\]
maps outer product null spaces to inner product null spaces
\begin{equation*}
\mathfrak{I}:{\bigwedge}^i V\rightarrow {\bigwedge}^{k-i} V,\quad\quad
{\bigwedge}^i V \ni \mathfrak{v}\mapsto \mathfrak{v}\cdot \mathfrak{I} \in{\bigwedge}^{k-i} V,
\end{equation*}
with $i\in\left\lbrace 1,\ldots,4 \right\rbrace $ and $k=5$ for the planar quadric geometric algebra. On the other hand
\[\mathfrak{I}^*:=e_2\wedge e_5 \wedge e_3\wedge e_6 \wedge(e_1+e_4)\]
maps dual elements to normal elements respectively inner product null spaces to outer product null spaces. There is no difference in left- or right multiplication with these $5$-blades. The result differs by the factor $-1$ and describes the same geometric entity.
\end{definition}
\noindent
With Def. \ref{DEF8} we get (see \cite{zamora:geometricalgebra})
\begin{equation*}
\mathds{NI}_G(\mathfrak{A})=\mathds{NO}_G(\mathfrak{A}\cdot \mathfrak{I}),\quad\quad
\mathds{NO}_G(\mathfrak{A})=\mathds{NI}_G(\mathfrak{A}\cdot \mathfrak{I}^*).
\end{equation*}
Note that dualization is realized with the inner product. Now we take a look at the inner product null space of grade-1 elements that are not embedded points. Therefore, at least one of the conditions \eqref{EQ22} or \eqref{EQ23} is not satisfied. Let
$\mathfrak{c}=-2a_1e_1+2a_2e_2-a_3e_3-2a_4e_4+2a_5e_5-a_6e_6$
be a general $1$-blade. The GIPNS results
\begin{align*}
\mathds{NI}_G(\mathfrak{c})&=\left\lbrace (x,y)^\mathrm{T} \in\mathds{R}^2 \mid\eta(x,y)\cdot \mathfrak{c}=0 \right\rbrace\\
&= \left\lbrace (x,y)^\mathrm{T} \in\mathds{R}^2 \mid a_1x^2+2a_2x+a_3+a_4y^2+2a_5y+a_6=0\right\rbrace .
\end{align*}
The GIPNS is a conic in principal position, because there is no term containing $xy$. Any conic has a coefficient matrix and is given by
\[\begin{pmatrix}
1 &x &y
\end{pmatrix}\begin{pmatrix}
a_{11}&a_{12}&a_{13}\\
a_{12}&a_{22}&a_{23}\\
a_{13}&a_{23}&a_{33}\\
\end{pmatrix}\begin{pmatrix}
1\\x\\y
\end{pmatrix}=0.\]
Therefore, we can define a bijection $\chi$ between those symmetric matrices which represent conics in principal position and $1$-blades by 
\begin{equation}\label{EQ37}
\begin{pmatrix}
a_0&a_2&a_5\\
a_2&a_1&0\\
a_5&0&a_4
\end{pmatrix}\mapsto 2a_1e_1-2a_2e_2+a_3e_3+2a_4e_4-2a_5e_5+a_6e_6.
\end{equation}
For the bijection \eqref{EQ37} we assume that $a_3:=\frac{1}{2}a_0$ and $a_6:=\frac{1}{2}a_0$. It would be sufficient to demand that $a_3+a_6=a_0$ to result in the same conic, because the constant value is equal to $a_3+a_6$. This does not change the GIPNS of the conic.
With Eq. \eqref{EQ37} embedded points can be interpreted as circles whose radii are equal to zero.\\

\noindent After dualization an inner product conic $\mathfrak{c}$ becomes an outer product conic $\hat{\mathfrak{c}}$ that is a four-blade and can be generated by the outer product of four embedded points. These four points lie on the conic because
\[\mathfrak{p}_i\in\mathds{NO}_G(\mathfrak{p}_1 \wedge \mathfrak{p}_2\wedge \mathfrak{p}_3 \wedge \mathfrak{p}_4), \mbox{ for $i=1,\ldots,4$}.\]
The natural question that arises is: Is there a way to classify conics in this model? For this purpose we study the incidence of the conics with the three additional ideal points. If a conic contains both ideal elements $e_3$ and $e_6$ it automatically contains also $e_3+e_6$. First we look at the entities $\mathfrak{a}\in{\bigwedge}^1 V$ that contain the ideal points $e_3,e_6$, and therefore, also $e_3+e_6$. Thus, we get the conditions
\begin{equation*}
\mathfrak{a}\cdot e_3=2a_1=0,\quad \mathfrak{a}\cdot e_6=2a_4=0.
\end{equation*}
Hence, $a_1$ and $a_4$ have to vanish. The corresponding algebra element has the form
\[\mathfrak{l}=2a_2e_2-a_3e_3+2a_5e_5-a_6e_6.\]
Its GIPNS is calculated by
\[\mathds{NI}_G(\mathfrak{l})=\left\lbrace (x,y)^T\in\mathds{R}^2 \mid 2a_2x+2a_5y+a_3+a_6=0 \right\rbrace.\]
Clearly, this entity describes an inner product line and every line passes through $e_3, e_6$, and $e_3+e_6$. An algebra element that contains just $e_3$ or $e_6$ is a parabola whose axis are parallel to the $x$-axis or the $y$-axis. An element that contains $e_3+e_6$, but neither $e_3$ nor $e_6$, is given by the condition $\mathfrak{a}\cdot (e_3+e_6)=2a_1+2a
_4=0$. This means $a_1=-a_4$ and the corresponding conic is an equilateral hyperbola, {\it i.e.}, the asymptotes enclose an angle of $90^\circ$. All other conics in principal axes position can be obtained by the wedge product of four embedded points or by using the bijection between conics and the algebra elements, cf. Eq. \eqref{EQ37}.
\begin{remark}
Note, that this description of conics also contains conics without real points.
\end{remark}
\noindent
In the most general case two-blades correspond to inner product point quadruples. This can be seen from
\begin{equation*}
\mathds{NI}_G(\mathfrak{a}\wedge \mathfrak{b})=\mathds{NI}_G(\mathfrak{a})\cap \mathds{NI}_G(\mathfrak{b}),
\end{equation*}
see \cite{perwass:geometricalgebra}. Therefore, two-blades represent all points belonging to both conics that are represented by the vectors. If two non-degenerate conics do not intersect, the corresponding two-blade represents a complex inner product point quadruple. Furthermore, we can see by dualization, that three-blades belong to outer product point quadruples. For two inner product lines $\mathfrak{l}_1,\mathfrak{l}_2$ the corresponding two-blade $\mathfrak{l}_1\wedge \mathfrak{l}_2$ represents an inner product pair of points, where one of the points is their affine intersection point and the other the point at infinity.
\begin{example}\label{EX1}
Let us generate a conic through four points. Therefore, we choose four points and embed them via \eqref{EQ21}.
\begin{align*}
P_1=(-1,0)^\mathrm{T}&\to\mathfrak{p}_1=e_1-e_2+\frac{1}{2}e_3+e_4,\\
P_2=(1,0)^\mathrm{T}\,\,\,\,\,&\to\mathfrak{p}_2=e_1+e_2+\frac{1}{2}e_3+e_4,\\
P_3=(0,-1)^\mathrm{T}&\to\mathfrak{p}_3=e_1+e_4-e_5+\frac{1}{2}e_6,\\ P_4=(-1,0)^\mathrm{T}&\to\mathfrak{p}_4=e_1+e_4+e_5+\frac{1}{2}e_6.\\
\end{align*}
The corresponding inner product representation is calculated by 
\[\mathfrak{c}=\mathfrak{I} \cdot (\mathfrak{p}_1 \wedge \mathfrak{p}_2 \wedge \mathfrak{p}_3 \wedge \mathfrak{p}_4)=4e_1-e_3+4e_4-e_6.\]
The GIPNS is given by
\[\mathds{NI}_G(\mathfrak{c})=\left\lbrace (x,y)^T\in\mathds{R}^2\mid -x^2+1-y^2=0\right\rbrace .\]
With the bijection \eqref{EQ37} we see easily that $\mathfrak{c}$ is the image of the conic given by the diagonal matrix diag$(1,-1,-1)$, {\it i.e.}, the unit circle centered at the origin.
\end{example}

\subsection{Transformations}
In this section we discuss transformations in this algebra. The Clifford algebra $\clifford_{(4,2,0)}$ corresponds to the quadratic space $\mathds{R}^{(4,2)}$, and therefore, the sandwich action of vectors represents reflections in hyperplanes in this space. It is clear that the transformations induce transformations that are not linear in the base space $\mathds{R}^2$ because the embedding $\eta$ is quadratic. From the last section we know that the geometric entity corresponding to vectors are axis aligned conics. Furthermore, a transformation acts via the sandwich operator and results again in a $k$-blade when applied to a $k$-blade $k=1,\ldots,4$. We begin with an example:
\begin{example}\label{EX2}
Let $\mathfrak{c}$ be the circle from Ex. \ref{EX1}
\[\mathfrak{c}=4e_1-e_3+4e_4-e_6\]
and let $\mathfrak{p}=e_1+e_2+\frac{1}{2}e_3+e_4+2e_5+2e_6$ be $\eta(1,2)^\mathrm{T}$. Applying the sandwich operator to $\mathfrak{p}$ results in
\[\mathfrak{p}'=\alpha(\mathfrak{c})\mathfrak{p}\mathfrak{c}^{-1}=5e_1+e_2-\frac{1}{2}e_3+5e_4+2e_5+e_6.\]
Now we check if this entity is still an embedded point. Therefore, the conditions \eqref{EQ22} and \eqref{EQ23} have to be checked. Condition \eqref{EQ22} is satisfied, but condition \eqref{EQ23} not
\[{\mathfrak{p}'}_x^2=6,\quad {\mathfrak{p}'}_y^2=-6.\]
Hence, $\mathfrak{p}'$ cannot be interpreted as embedded point. The GIPNS of $\mathfrak{p}'$ is given by
\[\mathds{NI}_G(\mathfrak{p}')=\left\lbrace (x,y)^\mathrm{T}\in\mathds{R}^2 \mid -\frac{5}{2}x^2+x-\frac{5}{2}y^2+2y-\frac{1}{2} =0\right\rbrace. \]
This represents a pair of complex lines intersecting in the real point $(\frac{1}{5},\frac{2}{5})^\mathrm{T}$.
\end{example}
\noindent
Ex. \ref{EX2} shows that in general a conic is mapped to another conic. We can not map a circle of radius zero to a circle of radius zero and define a mapping for points on this way. Thus, we study the action of the transformations applied to vectors that correspond to conics.
\begin{theorem}
A conic represented by the vector $\mathfrak{a}\in\bigwedge^1 V$ is pointwise fixed under the transformation induced by itself. Furthermore, these transformation are involutions.
\end{theorem}
\begin{proof}
First, we show that the conic corresponding to the transformation is fixed pointwise. Therefore, we look at the action of a general vector
\[\mathfrak{a}= -2a_1e_1+2a_2e_2-a_3e_3-2a_4e_4+2a_5e_5-a_6e_6\]
to itself, and find
\[\alpha(\mathfrak{a})\mathfrak{a}\mathfrak{a}^{-1}=\alpha(\mathfrak{a})=-\mathfrak{a}.\]
Multiplication with a homogeneous factor does not change the GIPNS. Thus, the result is the conic represented by $\mathfrak{a}$ again. To show that the points of the conic $\mathfrak{a}$ are fixed under the transformation induced by $\mathfrak{a}$, we examine the action of $\mathfrak{a}$ on the intersection points of the conic $\mathfrak{a}$ with all lines containing the point $(0,0)$. These lines are given by
\begin{align*}
\mathfrak{l}(x,y)&=\mathfrak{I}\cdot\left( \eta(0,0)\wedge \eta(x,y)\wedge e_3 \wedge e_6\right)\\
&=\mathfrak{I}\cdot\left( (e_1+e_4)\wedge (e_1+xe_2+\frac{1}{2}x_1^2e_3+e_4+ye_5+\frac{1}{2}y^2e_6)\wedge e_3 \wedge e_6\right)\\
&=2ye_2-2xe_5.
\end{align*}
The intersection of all these lines with the conic $\mathfrak{a}$ is represented by \begin{align*}
\mathfrak{l}(x,y)\wedge \mathfrak{a}=&-2a_3ye_{23}+4a_1ye_{12}-4a_4ye_{24}-4a_1xe_{15}+4(a_5y+a_2x)e_{25}\\
&-2a_3xe_{35}-2a_6ye_{26}-4a_4xe_{45}+2a_6xe_{56}.
\end{align*}
This two-blade represents the pair of common points of the conic and $\mathfrak{l}(x,y)$. The application of the transformation induced by $\mathfrak{a}$ to $\mathfrak{l}(x,y)\wedge \mathfrak{a}$ results in
\begin{align*}
\alpha(\mathfrak{a})(\mathfrak{l}(x,y)\wedge \mathfrak{a})\mathfrak{a}^{-1}&=-2a_3ye_{23}\!+\!4a_1ye_{12}\!-\!4a_4ye_{24}\!-\!4a_1xe_{15}\\
&+4(a_5y+a_2x)e_{25}\!-\!2a_3xe_{35}\!-\!2a_6ye_{26}\!-\!4a_4xe_{45}\!+\!2a_6xe_{56}.
\end{align*}
This shows, that all pairs of common points of the pencil of lines with the conic are fixed, and therefore, the whole conic is fixed pointwise. To see that the transformation is an involution we have to apply it twice to an arbitrary $k$-blade $\mathfrak{B}$
\[\alpha(\mathfrak{a})\alpha(\mathfrak{a})\mathfrak{B}\mathfrak{a}^{-1}\mathfrak{a}^{-1}=\alpha(\mathfrak{a}^2)\mathfrak{B}{(\mathfrak{a}^2)}^{-1}=\mathfrak{B}.\]
The last equality follows because $\mathfrak{a}^2$ is a real number.
\end{proof}
\noindent
Due to the fact that these transformations are represented as reflections with respect to hyperplanes in $\mathds{R}^{(4,2)}$, they are involutions and fix the corresponding hyperplane pointwise. This is the reason why we interpret these transformations as reflections or inversions with respect to conics. Furthermore, the whole group of transformations is generated by the action of vectors. Note that the image of a conic in principal position is always a conic in principal position in this model and that intersection point quadruples of a conic with the reflection conic stay fixed, no matter if the intersection points are real or complex.
\begin{remark}
The group of conformal transformations of a quadratic space $\mathds{R}^{(p,q)}$ can be described as the Pin group of a Clifford algebra $\clifford_{(p+1,q+1,0)}$, see \cite{porteous:cliffordalgebrasandtheclassicalgroups}. Therefore, the group of conformal transformations of the \emph{Minkowski space} $\mathds{R}^{(3,1)}$ is isomorphic to the group of inversions with respect to conics in principal position except for a translation. Especially for the planar quadric geometric algebra Q2GA we have the signature $(p,q,r)=(4,2,0)$ which is identical to the signature of the homogeneous model where we choose Lie's quadric as metric quadric. Thus, the Pin group of Q2GA is isomorphic to the group of Lie transformations.
\end{remark}
\subsection{Effect on Lines and Points}
In Ex. \ref{EX2} we have seen that a circle with radius zero is mapped to a pair of complex conjugate lines intersecting in a real point. Therefore, we have to search for a better description of points in this model. One way to describe points as two-blades is to examine the intersection of two lines. If we take just the affine point of intersection, we can define an embedding of the affine plane as points of intersection of pairs of lines. Therefore, we take two lines through a given point $(x,y)^\mathrm{T}\in\mathds{R}^2$. We define this point to be the point of the line parallel to the $x$-axis and the line parallel to the $y$-axis.
\begin{align}
\mathfrak{p}&=(\mathfrak{I}\cdot (\eta(x,y)\wedge\eta(x,0)\wedge e_3\wedge e_6))\wedge (\mathfrak{I}\cdot (\eta(x,y)\wedge\eta(0,y)\wedge e_3\wedge e_6))\nonumber\\
&= -y^2xe_{23}-2yx e_{25}-yx^2 e_{35}+yx^2 e_{56}-y^2x e_{26}\nonumber\\
&= 2 e_{25}+x( e_{35}-e_{56})+y(e_{23}+ e_{26}).
\label{EQ41}
\end{align}
Note that this element represents a pair of points, since lines meet also in the point $\infty$. We parametrize affine lines as the sets of all lines passing through two points $p_1=(x_1,y_1)^\mathrm{T}$ and $p_2=(x_2,y_2)^\mathrm{T}$. The inner product line is derived by
\begin{align}
\mathfrak{l}(p_1,p_2)&=\mathfrak{I}\cdot\left( \eta(p_1)\wedge \eta(p_2)\wedge e_3\wedge e_6\right)\nonumber\\
&= -2(y_1-y_2) e_2+(x_1y_2-y_1x_2) e_3\nonumber\\
&+2(x_1-x_2) e_5+(x_1y_2-y_1x_2)e_6.\label{EQ42}
\end{align}
The GIPNS of this line is determined by
\[\mathds{NI}_G(\mathfrak{l}(p_1,p_2))=\left\lbrace(x,y)\in\mathds{R}^2\mid (x_1-x_2)y+(y_2-y_1)x+(y_1x_2-x_1y_2)=0 \right\rbrace.\]
\begin{theorem}
The image of a line under an inversion with respect to a conic represented by the non-null vector $\mathfrak{a}\in{\bigwedge}^1 \clifford_{(4,2,0)}$ is a conic. Moreover, for non-degenerate conics this conic is the image of $\mathfrak{a}$ under an affine transformation, {\it i.e.}, translation and scalar multiplication.
\end{theorem}
\begin{proof}
To show this we concentrate on conics with no terms in $x$ and $y$. We can do this because we are just interested in the type of the image conic. Furthermore, we can perform translations by two reflections in parallel lines, and thus, we can carry over the results from the principal position to an arbitrary position. Furthermore, we just show this theorem for non degenerate conics. A conic with no terms in $x$ or $y$ is given by
\[\mathfrak{a}=2a_1e_1+\frac{1}{2}a_0e_3+2a_4 e_4+\frac{1}{2}a_0e_6.\]
Since we are interested in real conics, the coefficients $a_0,a_1$, and $a_4$ are not allowed to have the same sign. In the sequel we assume that the conic is real. Now we can look at the matrix of the conic
\[\mathrm{M}=\begin{pmatrix}
1&0&\\0&\frac{a_1}{a_0}&0\\0&0&\frac{a_4}{a_0}
\end{pmatrix}.\]
The image of the set of lines \eqref{EQ42} is calculated by
\begin{align*}
\alpha(\mathfrak{a})\mathfrak{l}(p_1,p_2)\mathfrak{a}^{-1}&=-\frac{4a_1(x_1y_2-y_1x_2)}{a0}e_1-2(y_1-y_2)e_2\\
&-\frac{4a_4(x_1y_2-y_1x_2)}{a_0}e_4+2(x_1-x_2)e_5.
\end{align*}
The coefficient matrix of the corresponding conic is given by
\[\mathrm{N}(p_1,p_2)=\frac{1}{a_0}\begin{pmatrix}
0& a_0(y_2-y_1)&a_0(x_1-x_2)\\
a_0(y_2-y_1) & 2c_1(x_1y_2-y_1x_2)&0\\
a_0(x_1-x_2)&0&2a_4(x_1y_2-y_1x_2)
\end{pmatrix}.\]
From this representation we see immediately  that lines through the center of the conic are fixed, but not pointwise.
In order to transform this matrix to diagonal form we apply the transformation
\[p\mapsto\underbrace{\begin{pmatrix}
1&0&0\\\alpha&1&0\\\beta&0&1
\end{pmatrix}}_{=:\mathrm{T}}p,\mbox{ with $\alpha=-\frac{a_0(y1-y2)}{2a_1(x_1y_2-y_1x_2)}$ and $\beta=\frac{a_0(x1-x2)}{2a_4(x_1y_2-y_1x_2)}$}.\]
Here $p=(1,x,y)^\mathrm{T}\in\mathds{R}^3$ is a point in the projective plane. The action of this coordinate transformation applied to the coefficient matrix of the conic yields
\begin{align*}
\mathrm{N}'(p_1,p_2)&=\mathrm{T}^{-\mathrm{T}}\mathrm{N}(p_1,p_2)\mathrm{T}^{-1}\\
&=\begin{pmatrix}
-\frac{a_0^2(a_4(y_1-y_2)^2+a_1(x_1-x_2)^2)}{2a_4a_1(x_1y_2-y_1x_2))}&0&0\\
0&\frac{2 a_1(x_1y_2-y_1x_2)}{a_0}&0\\
0&0&\frac{2 a_4(x_1y_2-y_1x_2)}{a_0}
\end{pmatrix}.
\end{align*}
If we look at the affine part of the conic, we see that this results in
\[\mathrm{N}'(p_1,p_2)=\begin{pmatrix}
1&0&0\\0&ka_1&0\\0&0&ka_4
\end{pmatrix},\mbox{ with $k=-\frac{4a_1^2(x_1y_2-y_1x_2)^2a_4}{a_0^2(a_4(y_1-y_2)^2+a_1(x_1-x_2)^2)}$}.\]
Therefore, the image is identical to the conic corresponding to $\mathfrak{a}$ except for a translation $\mathrm{T}$ and a scaling $k$.
Furthermore, lines are mapped to real conics.
\end{proof}
\noindent
To illustrate this, Fig. \ref{FIG1} shows the reflections of three intersecting lines with respect to a circle, an ellipse, a parabola, and a hyperbola. The inversion conic $\mathfrak{a}_i,\,i=1,\ldots,4$ is shown in red while the pairs (line, image of the line) are presented in another colour (but the same). The inversion conics are from left to right and from up to down given by 
\begin{equation*}
\mathfrak{a}_1:x^2+y^2=1,\quad
\mathfrak{a}_2:\frac{25}{16}x^2+\frac{16}{25}y^2=1,\quad
\mathfrak{a}_3:x^2-y=1,\quad
\mathfrak{a}_4:x^2-\frac{3}{4}y^2=1.
\end{equation*}
\begin{figure}[h]
\begin{center}
\begin{overpic}[scale=0.24]{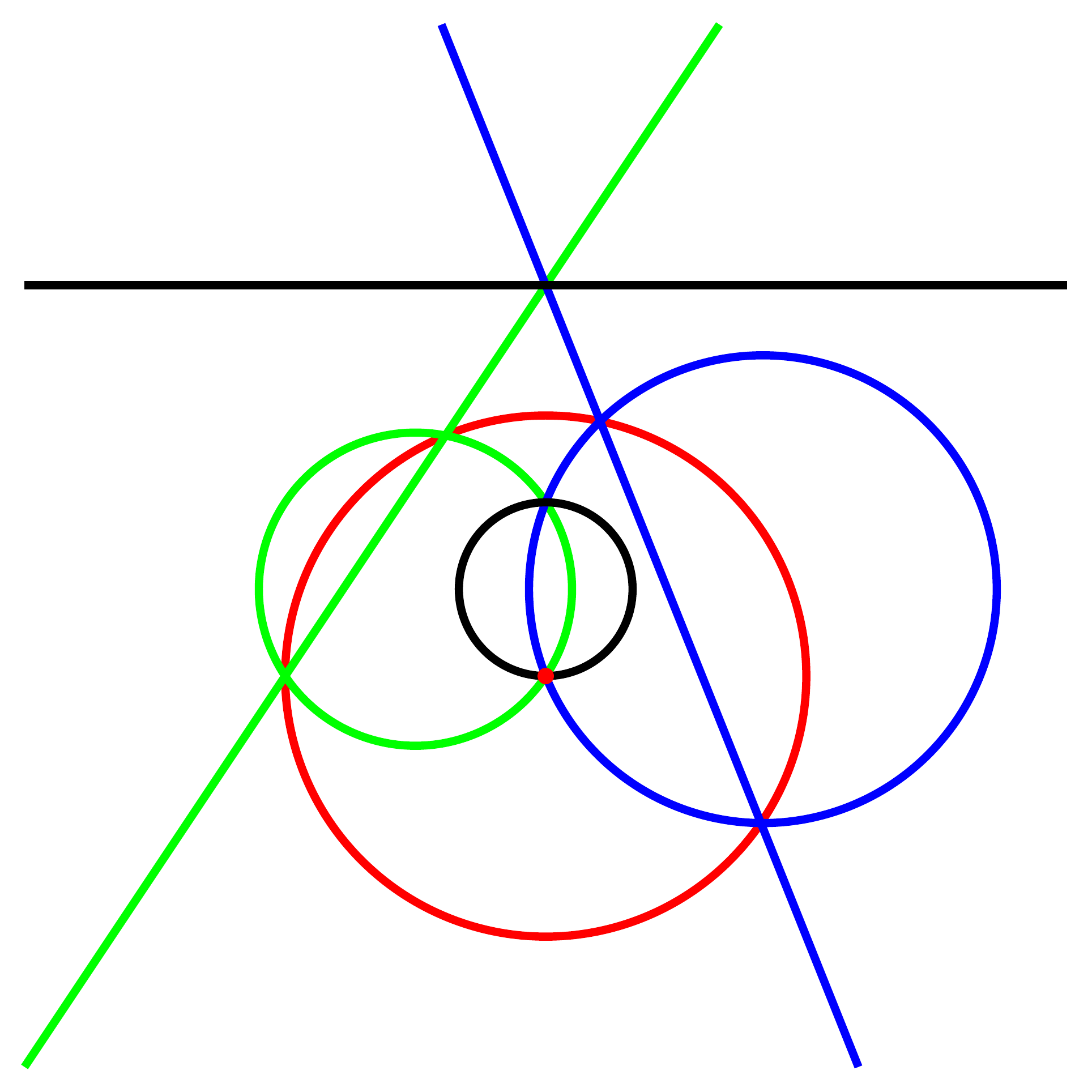}
\put(40,10){$\mathfrak{a}_1$}
\put(-2,5){$\mathfrak{l}_1$}
\put(18,50){$\mathfrak{c}_1$}
\put(17,76){$\mathfrak{l}_2$}
\put(35,45){$\mathfrak{c}_2$}
\put(70,5){$\mathfrak{l}_3$}
\put(85,63){$\mathfrak{c}_2$}
\put(47.5,33){$o$}
\end{overpic}
\hspace{4em}
\begin{overpic}[scale=0.24]{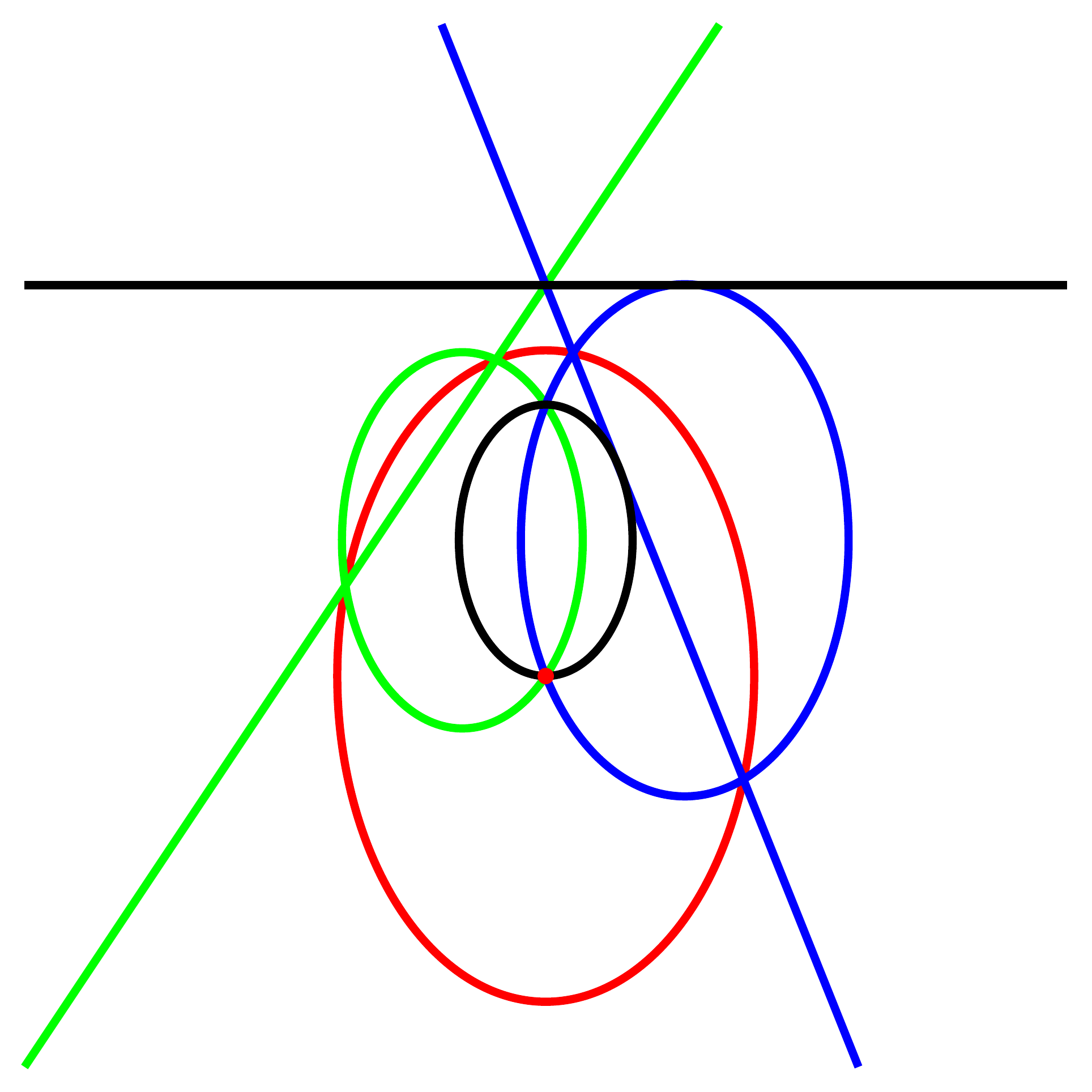}
\put(40,5){$\mathfrak{a}_2$}
\put(-2,5){$\mathfrak{l}_1$}
\put(25,55){$\mathfrak{e}_1$}
\put(17,76){$\mathfrak{l}_2$}
\put(36,42){$\mathfrak{e}_2$}
\put(70,5){$\mathfrak{l}_3$}
\put(77,62){$\mathfrak{e}_3$}
\put(47.5,33){$o$}
\end{overpic}\\
\begin{overpic}[scale=0.24]{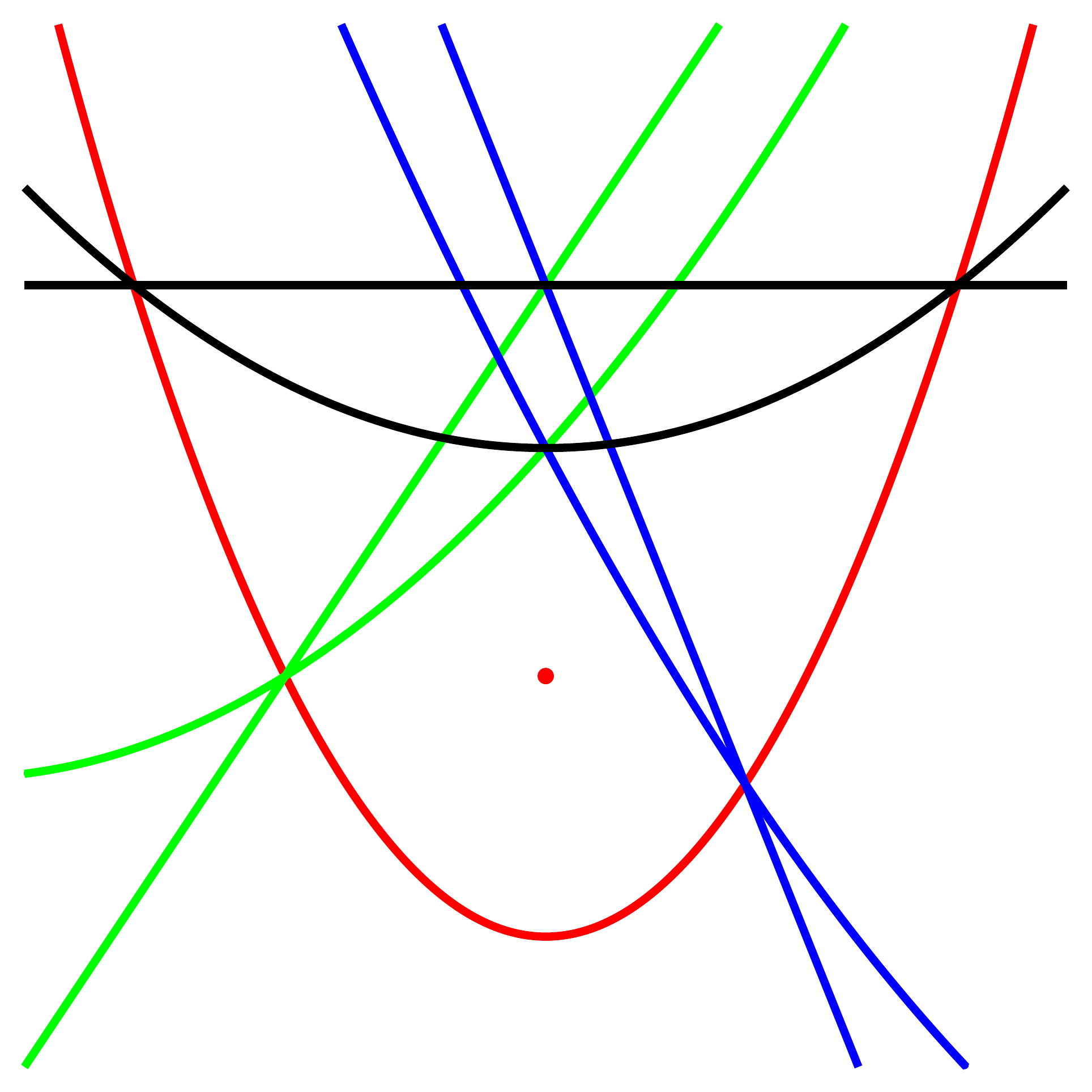}
\put(40,11){$\mathfrak{a}_3$}
\put(-2,5){$\mathfrak{l}_1$}
\put(5,33){$\mathfrak{p}_1$}
\put(17,76){$\mathfrak{l}_2$}
\put(23,60){$\mathfrak{p}_2$}
\put(70,5){$\mathfrak{l}_3$}
\put(80,13){$\mathfrak{p}_3$}
\put(47.5,33){$o$}
\end{overpic}
\hspace{4em}
\begin{overpic}[scale=0.24]{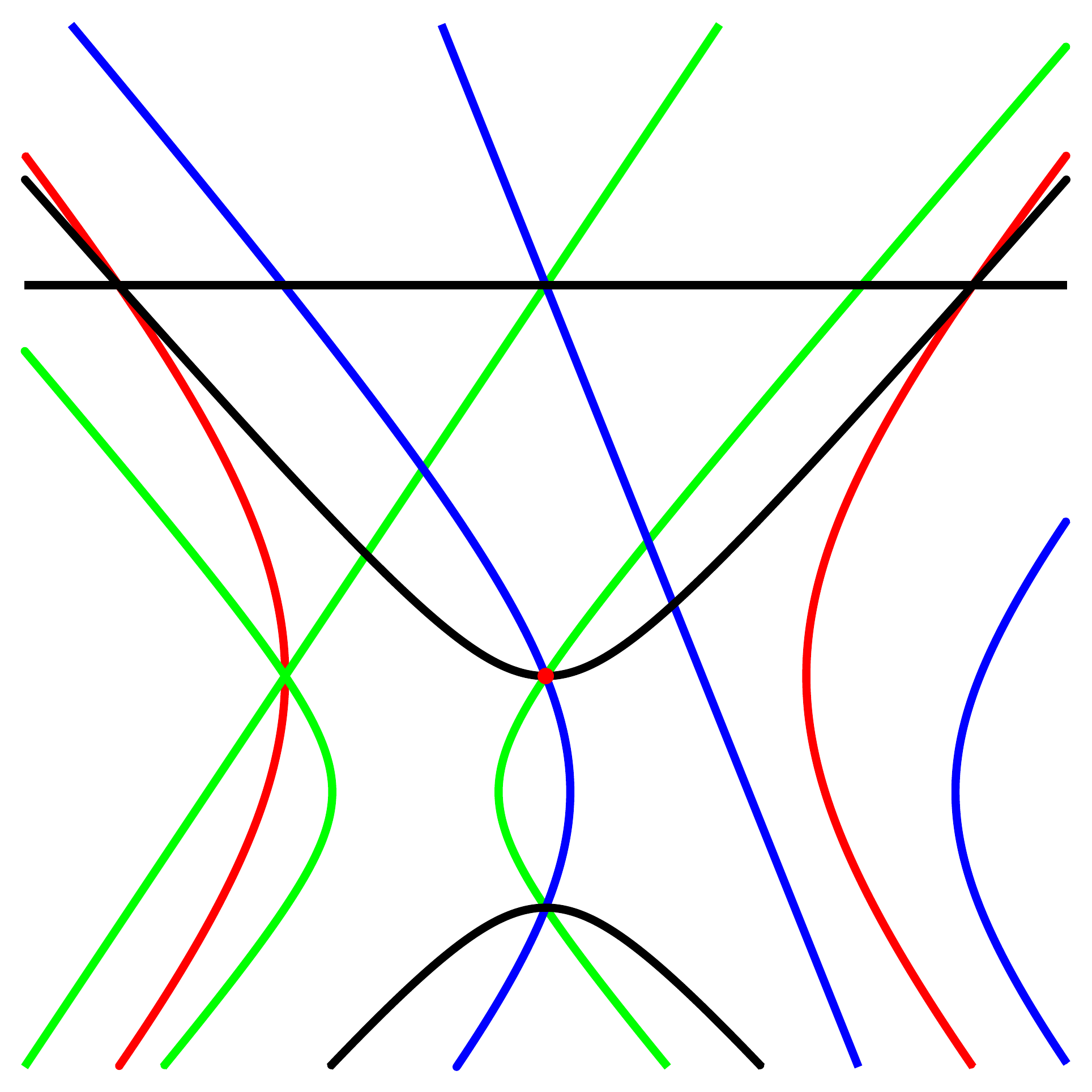}
\put(75,35){$\mathfrak{a}_4$}
\put(-2,5){$\mathfrak{l}_1$}
\put(0,60){$\mathfrak{h}_1$}
\put(17,76){$\mathfrak{l}_2$}
\put(35,37){$\mathfrak{h}_2$}
\put(70,5){$\mathfrak{l}_3$}
\put(3,90){$\mathfrak{h}_3$}
\put(47.5,33){$o$}
\end{overpic}
\caption{Inversions with respect to conics}
\label{FIG1}
\end{center}
\end{figure}
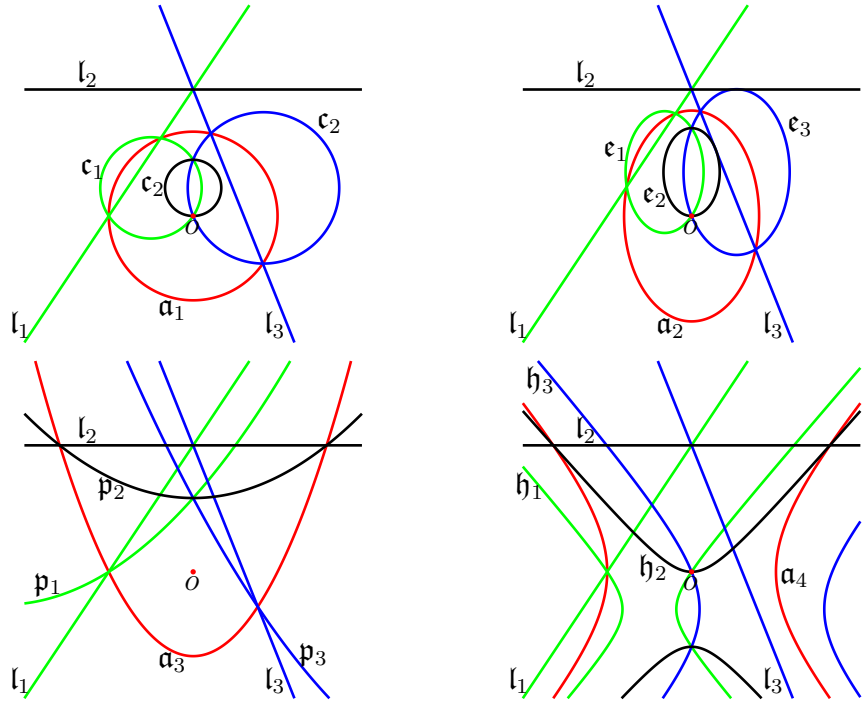

\subsection{Subgroups}
In this section we examine some subgroups that are embedded naturally in the Pin and the Spin group of Q2GA.
\paragraph{Rotation}
First, we concentrate on the group that is generated by inversions with respect to lines passing through the origin. Therefore, we study the action of these mappings applied to points embedded via \eqref{EQ41}. Two lines through the origin may be represented by
\begin{align*}
\mathfrak{l}_1&=\mathfrak{I}\cdot\left( \eta(0,0)\wedge\eta(\cos\varphi,\sin\varphi)\wedge e_3\wedge e_6 \right),\\
\mathfrak{l}_2&=\mathfrak{I}\cdot\left( \eta(0,0)\wedge\eta(\cos\psi,\sin\psi)\wedge e_3\wedge e_6 \right).
\end{align*}
Furthermore, we are interested in orientation preserving transformations, {\it i.e.}, elements from the Spin group. The composition of two reflections in $\mathfrak{l}_1$ and $\mathfrak{l}_2$ is given by their geometric product
\begin{align}
\mathfrak{l}_1\mathfrak{l}_2&=(\sin\psi\sin\varphi+\cos\psi\cos\varphi)
+(\cos\psi\sin\varphi-\sin\psi\cos\varphi)e_{25}\nonumber\\
\intertext{and with the addition theorems for sine and cosine we conclude}
\mathfrak{l}_1\mathfrak{l}_2&=\cos(\varphi-\psi)+\sin(\varphi-\psi)e_{25}\label{EQ43}.
\end{align}
\FloatBarrier
\noindent The square of $e_{25}$ is $-1$. This means that the consecutive reflection in two lines through the origin results in an algebra element that can be interpreted as complex number with norm equal to $1$. It is a well-known result that rotations in the plane can be described by normed complex numbers. Let us look at the action of such an element applied to a point that is described by \eqref{EQ41}. Let $\mathfrak{R}=\cos\varphi+\sin\varphi e_{25}$ be an element in the form of Equation \eqref{EQ43} and $\mathfrak{p}=2 e_{25}+x_0( e_{35}-e_{56})+y_0(e_{23}+ e_{26})$ a point of the form \eqref{EQ41}. We compute
\begin{align*}
\mathfrak{p}'=&\alpha(\mathfrak{R})\mathfrak{p}\mathfrak{R}^{-1}=\mathfrak{R}\mathfrak{p}\mathfrak{R}^{-1}\\
=&2 e_{25}+(-2\sin\varphi\cos\varphi x_0+2\cos(\varphi)^2 y_0-y_0)(e_{23}+e_{26})\\
+&(2\sin\varphi\cos\varphi y_0+2\cos(\varphi)^2 x_0-x_0) (e_{35}-e_{56}).
\end{align*}
The GIPNS of this entity can be computed or we can simply read the coordinates of the image point.
\begin{align*}
x&=2\cos(\varphi)^2 x_0+2\cos\varphi\sin\varphi y_0-x_0=\cos (2\varphi) x_0 +\sin (2\varphi) y_0 ,\\
y&=2\cos(\varphi)^2 y_0-2\cos\varphi\sin\varphi x_0-y_0=\cos (2\varphi) y_0 -\sin (2\varphi) x_0.
\end{align*}
Therefore, we can see that this transformation, indeed, is a rotation about the origin with the rotation angle $2\varphi$. So these elements constitute a double cover of the group $\mathrm{SO}(2)$.
\begin{remark}
From the fact that vectors are mapped to vectors by a reflection in a line it follows that axis aligned conics have to be mapped to axis aligned conics. Therefore, these mappings can not be interpreted pointwise for conics.
\end{remark}
\paragraph{Translations}
Now we aim at the group of planar Euclidean displacements $\mathrm{SE}(2)$. Therefore, we show that two consecutive reflections in parallel lines result in a translation. The group of planar Euclidean displacements can be generated as the semi-direct product of $\mathrm{SO}(2)$ and $\mathrm{T}(2)$, which describes the abelian translation group. Let $\mathfrak{l}_1$ and $\mathfrak{l}_2$ be two parallel lines and let $t_1, t_2$ be their distances from the origin. The lines are given by
\begin{align*}
\mathfrak{l}_1(\varphi,t_1)&=2\sin\varphi e_2-t_1 e_3-2\cos\varphi e_5-t_1 e_6,\\
\mathfrak{l}_2(\varphi,t_2)&=2\sin\phi e_2-t_2 e_3-2\cos\varphi e_5-t_2 e_6.
\end{align*}
The composition can be expressed with the geometric product as
\begin{align}
\mathfrak{T}(\varphi,t_1,t_2)&=\mathfrak{l}_1(\varphi,t_1)\mathfrak{l}_2(\varphi,t_2)\nonumber\\
&=2+(t_1-t_2)\sin\varphi (e_{23}+e_{26})+(t_1-t_2)\cos\varphi (e_{35}-e_{56}).\label{EQ44}
\end{align}
Applying the sandwich operator to a point $\mathfrak{p}$ results in
\begin{align*}
\alpha(\mathfrak{T})\mathfrak{p}\mathfrak{T}^{-1}&=\mathfrak{T}\mathfrak{p}\mathfrak{T}^{-1}\\
&=2e_{25}+(y_0-2t_2\cos\varphi+2t_1\cos\varphi)(e_{23}+e_{26})\\
 &+(x_0-2t_1\sin\varphi+2t_2\sin\varphi) (e_{35}-e_{56}).
\end{align*}
The image is determined by
\begin{equation*}
x=x_0-\sin\varphi(2(t_1-t_2)),\quad\quad y=y_0+\cos\varphi(2(t_1-t_2)).
\end{equation*}
Therefore, the transformation is a translation in the direction normal to the given lines $\mathfrak{l}_1$ and $\mathfrak{l}_2$.
\paragraph{The Group of planar Euclidean Displacements}
Translations and rotations about the origin generate the entire group of planar Euclidean displacements $\mathrm{SE}(2)$. Furthermore, we can now examine the group that is generated by rotations and translations as a subgroup of the Spin group. These algebra elements have the form
\FloatBarrier
\[a_0 + a_1 e_{25}+a_2 (e_{23}+e_{26})+a_3 (e_{35}-e_{56}).\]
The multiplication table of the geometric product for the generators $e_0,e_{25},e_{23}+e_{26},e_{35}-e_{56}$ is given in Table \ref{Table1}.
\begin{table}[hbt]
\begin{center}
\begin{tabular}{|c|c|c|c|c|}
\hline 
 & $1$ & $e_{25}$ & $e_{23}+e_{26}$ & $e_{35}-e_{56}$ \\ 
\hline
$1$ & $1$ & $e_{25}$ & $e_{23}+e_{26}$ & $e_{35}-e_{56}$ \\ 
\hline
$e_{25}$ & $e_{25}$ & -1 & $e_{35}-e_{56}$ & $-(e_{23}+e_{26})$ \\ 
\hline
$e_{23}+e_{26}$ & $e_{23}+e_{26}$ & $-(e_{35}-e_{56})$ & 0 & 0 \\ 
\hline
$e_{35}-e_{56}$ & $e_{35}-e_{56}$ & $e_{23}+e_{26}$ & 0 & 0 \\ 
\hline 
\end{tabular}
\caption{Multiplication table of planar displacements in $\clifford_{(4,2,0)}$}
\label{Table1}
\end{center}
\end{table}
Hence, this is indeed a subgroup of the Spin group. Furthermore, it is isomorphic to a subgroup of the multiplicative group of dual quaternions, called planar dual quaternions. We can define a bijection by
\[e_{25}\mapsto \mathbf{i},\quad (e_{23}+e_{26}) \mapsto \epsilon\mathbf{j},\quad (e_{35}-e_{56}) \mapsto \epsilon\mathbf{k}.\]

\begin{remark}
If we restrict ourself to reflections in lines and circles, we are able to describe the group of conformal transformations of the plane.
\end{remark}

\paragraph{Inversions applied to Points}
In this section we study the action of reflections with respect to a conic in principal position (centered at the origin) on points. The points are embedded as points of intersection of two lines, as discussed in the previous section. Furthermore, we have to note that the transformations map point pairs to point pairs. Hence, the pair of points of intersection of two lines (the affine and the ideal point) are mapped to a pair of points. All lines pass through $\infty$ and so the image of every line must pass through the image of this point. The generalization to conics not in principal position is obtained by the application of a coordinate transformation. The inversion conic is given by
\[\mathfrak{a}=\frac{1}{2}c_0 (e_3+e_6) +2 c_1 e_1 +2c_2 e_4.\]
A point is represented as intersection of two lines (see \eqref{EQ41} by
\[\mathfrak{p} = 2 e_{25}+y_0 (e_{23}+e_{26})+x_0(e_{35}-e_{56}).\]
Applying the sandwich operator to the point results in
\begin{equation*}
\mathfrak{p}'=\alpha(\mathfrak{a})\mathfrak{p}\mathfrak{a}^{-1}= -\frac{2y_0c_1}{c_0}e_{12}+\frac{2x_0c_1}{c_0}e_{15}-2e_{25}+\frac{2y_0c_2}{c_0}e_{24}+\frac{2x_0c_2}{c_0}e_{45}.
\end{equation*}
This is not of the form \eqref{EQ41}, and therefore, it is not the representation of the intersection of two lines. The GIPNS is calculated
\begin{align*}
\mathds{NI}_G(\mathfrak{p}')=&\left\lbrace (x,y)^\mathrm{T} \in\mathds{R}^2\mid -2c_1(-x_0y+xy_0)e_1-(y_0c_1x^2+2yc_0+y_0c_2y^2)e_2\right. \\
&\left. -2c_2(-x_0y+xy_0)e_4+(x_0c_1x^2+x_0c_2y^2+2xc_0)e_5=0\right\rbrace 
\end{align*}
The solution set of this GIPNS can be written as
\begin{equation*}
x=-\frac{2c_0x_0}{c_1x_0^2+c_2 y_0^2},\quad\quad
y=-\frac{2c_0y_0}{c_1x_0^2+c_2 y_0^2}.
\end{equation*}
Note that we excluded the solution $x=0,y=0$, that is the image of $\infty$ under the inversion.
\begin{remark}
Inversion with respect to conics that are not in principal position can be performed by the composition of an inversion and a rotation. Note, that this rotation has to be applied pointwise.
\end{remark}
\subsection{Generalization to higher Dimensions}
The main advantage of this geometric algebra model is its flexibility. It is no problem to change the dimension. We discuss the model for the $n$-dimensional case and we show some examples for the three-dimensional case. We start with a real vector space of dimension $n$. For each axis we use a conformal embedding. Therefore, the dimension of the geometric algebra is $2^{3n}$and its quadratic form is given by 
\[\mathrm{Q}= \underbrace{\begin{pmatrix}
\mathrm{D} &   &  & \\
  & \mathrm{D} &  & \\
  &   &  \ddots & \\
  &   &   &  \mathrm{D}
\end{pmatrix}}_{n\mbox{-times}},\quad \mathrm{D}=\begin{pmatrix}
0 & 0 & -1\\ 0 & 1 & 0\\ -1 & 0 &0
\end{pmatrix}.
\]
The embedding $\eta$ is realized
by
\begin{align}
\eta:\mathds{R}^n&\to{\bigwedge}^1 V,\nonumber\\
(x_1,\ldots,x_n)&\mapsto e_1+x_1 e_2+\frac{1}{2}x_1^2e_3+\ldots +e_{3n-2}+x_n e_{3n-1}+\frac{1}{2}x_n^2 e_{3n}.\label{EQ51}
\end{align}
The conditions for an embedded point \eqref{EQ22} and \eqref{EQ23} generalize to
\[\eta(P)^2=0,\quad\eta(P)_{x_1}^2=0,\quad\eta(P)_{x_2}^2=0,  \ldots\quad \eta(P)_{x_n}^2=0.\]
We define analogue to Def. \ref{DEF8}:
\begin{definition}
The blade $\mathfrak{I}$ that maps outer product null spaces to inner product null spaces is defined by
\[\mathfrak{I}=\bigwedge\limits_{\substack{i=1\\i\mod 3=2}}^n e_i \wedge \bigwedge\limits_{\substack{j=1\\j\mod 3=1}}^n e_j\wedge \sum\limits_{\substack{k=1\\k\mod 3=0}}^n e_k .\]
Inner product null spaces can be mapped to outer product null spaces with the blade
\[\mathfrak{I}^*:=\bigwedge\limits_{\substack{i=1\\i\mod 3=2}}^n e_i \wedge \bigwedge\limits_{\substack{j=1\\j\mod 3=0}}^n e_j\wedge \sum\limits_{\substack{k=1\\k\mod 3=1}}^n e_k.\]
\end{definition}
\noindent
Grade-1 elements correspond to inner product axis aligned hyperquadrics. As the dimension is growing, the number of objects that can be represented grows, too. Blades of grade $k, (k\leq n)$ correspond to the intersection of $k$ hyperquadrics.

\paragraph{Quadrics in $3$ Dimensions}
To construct the quadric geometric algebra for the three-dimensional space we use the quadratic space $\mathds{R}^{(6,3)}$ given by the nine-dimensional real vector space $\mathds{R}^9$ together with the quadratic form
\[\mathrm{Q}= \begin{pmatrix}
\mathrm{D} &   & \\
  & \mathrm{D} & \\
  &   &  \mathrm{D}
\end{pmatrix},\quad \mathrm{D}=\begin{pmatrix}
0 & 0 & -1\\ 0 & 1 & 0\\ -1 & 0 &0
\end{pmatrix}.
\]
For three dimensions the embedding $\eta$, see Eq \eqref{EQ51}, has the following form
\begin{align}
\eta:\mathds{R}^3&\to {\bigwedge}^1{V},\nonumber\\
(x,y,z)^\mathrm{T}&\mapsto e_1+xe_2+\frac{1}{2}x^2 e_3+e_4+ye_5+\frac{1}{2}y^2 e_6+e_7+ze_8+\frac{1}{2}z^2 e_9\nonumber.
\end{align}
The conditions for an embedded point \eqref{EQ22} and \eqref{EQ23} generalize to
\[\eta(P)^2=0,\quad \eta(P)_x^2=0,\quad\eta(P)_y^2=0,\quad \eta(P)_z^2=0.\]
Moreover, the blades $\mathfrak{I}$ and $\mathfrak{I}^*$ are given by
\begin{align}
\mathfrak{I}&=(e_2\wedge e_5\wedge e_8) \wedge (e_1\wedge e_4\wedge e_7)\wedge (e_3+e_6+e_9),\nonumber\\
\mathfrak{I}^*&=(e_2\wedge e_5\wedge e_8) \wedge (e_3\wedge e_6\wedge e_9)\wedge (e_1+e_4+e_7)\nonumber.
\end{align}
The corresponding geometric algebra has dimension $2^9=512$. Any quadric in principal position except for translation in $\mathds{R}^3$ is uniquely determined by six values. This can be seen from the symmetric matrix of the equation of the quadric that has, in general, ten free entries. The fact that we are treating quadrics in principal position reduces the number of free entries to seven. Furthermore, this matrix representation is homogeneous, and therefore, we have six degrees of freedom. In analogy to Eq. \eqref{EQ37} we obtain a bijection $\chi$ that is defined as 
\begin{equation*}
\begin{pmatrix}
c_0&c_2&c_5&c_8\\
c_2 & c_1 &0 &0\\
c_5 & 0 & c_4 & 0\\
c_8 &0 & 0 & c_7
\end{pmatrix}\mapsto \mathfrak{q},
\end{equation*}
with
\[\mathfrak{q}=2c_1e_1\!-\!2c_2e_2\!+\!\frac{1}{3}c_0e_3\!+\!2c_4e_4\!-\!2c_5e_5+\!\frac{1}{3}c_0e_6 \!+\!2c_7e_7\!-\!2c_8e_8\!+\!\frac{1}{3}c_0e_9.\]
As we did for the planar case, we can now pay our attention on the intersection of three planes in order to get a pair of points containing one affine and one ideal point. We choose these planes to be parallel to the coordinate planes and passing through a given point $P=(x_0,y_0,z_0)^\mathrm{T}$. Expressed in terms of the quadric geometric algebra $\clifford_{(6,3,0)}$ we get
\begin{align*}
\mathfrak{p}&=(\eta(x_0,y_0,z_0)^\mathrm{T}\wedge\eta(0,0,z_0)^\mathrm{T}\wedge\eta(x_0,0,z_0)^\mathrm{T}\wedge e_3\wedge e_6\wedge e_9)\cdot \mathfrak{I}\\
& \wedge (\eta(x_0,y_0,z_0)^\mathrm{T}\wedge \eta(0,y_0,0)^\mathrm{T}\wedge\eta(x_0,y_0,0)^\mathrm{T}\wedge e_3\wedge e_6\wedge e_9)\cdot \mathfrak{I}\\
& \wedge (\eta(x_0,y_0,z_0)^\mathrm{T}\wedge\eta(x_0,0,0)^\mathrm{T}\wedge\eta(x_0,y_0,0)^\mathrm{T}\wedge e_3\wedge e_6\wedge e_9)\cdot \mathfrak{I}\\
&=3e_{258}+(e_{358}-e_{568}+e_{589})x_0+({e_{238}+e_{268}-e_{289}})y_0\\
&+(-e_{235}+e_{256}+e_{259})z_0.
\end{align*}
\begin{remark}
A representation of the group of Euclidean displacements $\mathrm{SE}(3)$ can be obtained by studying the composition of reflections in planes. Planes correspond to vectors that are obtained by $(\eta P_1\wedge\eta P_2\wedge\eta P_3\wedge e_3 \wedge e_6 \wedge e_9)\cdot\mathfrak{I}$.
\end{remark}
\noindent Now we define an inner product inversion quadric with $P_1=( \frac{9}{10},0,0)^\mathrm{T},\,P_2=( -\frac{9}{10},0,0)^\mathrm{T},$ $P_3=( 0,\frac{3}{4},0)^\mathrm{T},\,P_4=( 0,-\frac{3}{4},0)^\mathrm{T},\,P_5=( 0,0,\frac{5}{4})^\mathrm{T},\,P_6=( 0,0,-\frac{5}{4})^\mathrm{T}$
\begin{equation*}
\mathfrak{a}=(\eta P_1\wedge \eta P_2\wedge \eta P_3 \wedge \eta P_4\wedge \eta P_5\wedge \eta P_6)\cdot \mathfrak{I}=\frac{25}{9}e_1-\frac{3}{8}e_3+4e_4-\frac{3}{8}e_6+\frac{36}{25}e_7-\frac{3}{8}e_9.
\end{equation*}
The GIPNS of $\mathfrak{a}$ is given by
\[\mathds{NI}_G(\mathfrak{a})=\left\lbrace (x,y,z)^\mathrm{T}\in\mathds{R}^3 \Big|\, \frac{100}{81}x^2+\frac{16}{9}y^2+\frac{16}{25}z^2-1=0\right\rbrace.\]
In Fig. \ref{FIG2} (left) we can see the inversion of the unit cube $[-1,1]^3$ with respect to the ellipsoid defined by $\mathfrak{a}$. The image of every face of the cube, {\it i.e.}, of every plane is an ellipsoid passing through the origin.
\noindent
The action of the inversion given by $\mathfrak{a}$ on pairs of points can be written as
\begin{equation*}
f(x,y,z)=\frac{45^2}{2^2(30^2 y^2+25^2x^2+18^2z^2)}\left(x,y,z  \right) ^\mathrm{T}.
\end{equation*}
Note that the point $\infty$ is mapped to the origin and that the origin is mapped to $\infty$. Therefore, we have a map from $\mathds{R}^3\backslash\left\lbrace 0 \right\rbrace \to \mathds{R}^3$.
\begin{figure}[t]
        \begin{center}
        \begin{minipage}[b]{0.50\textwidth}
                \centering
                \begin{overpic}[width=\textwidth]{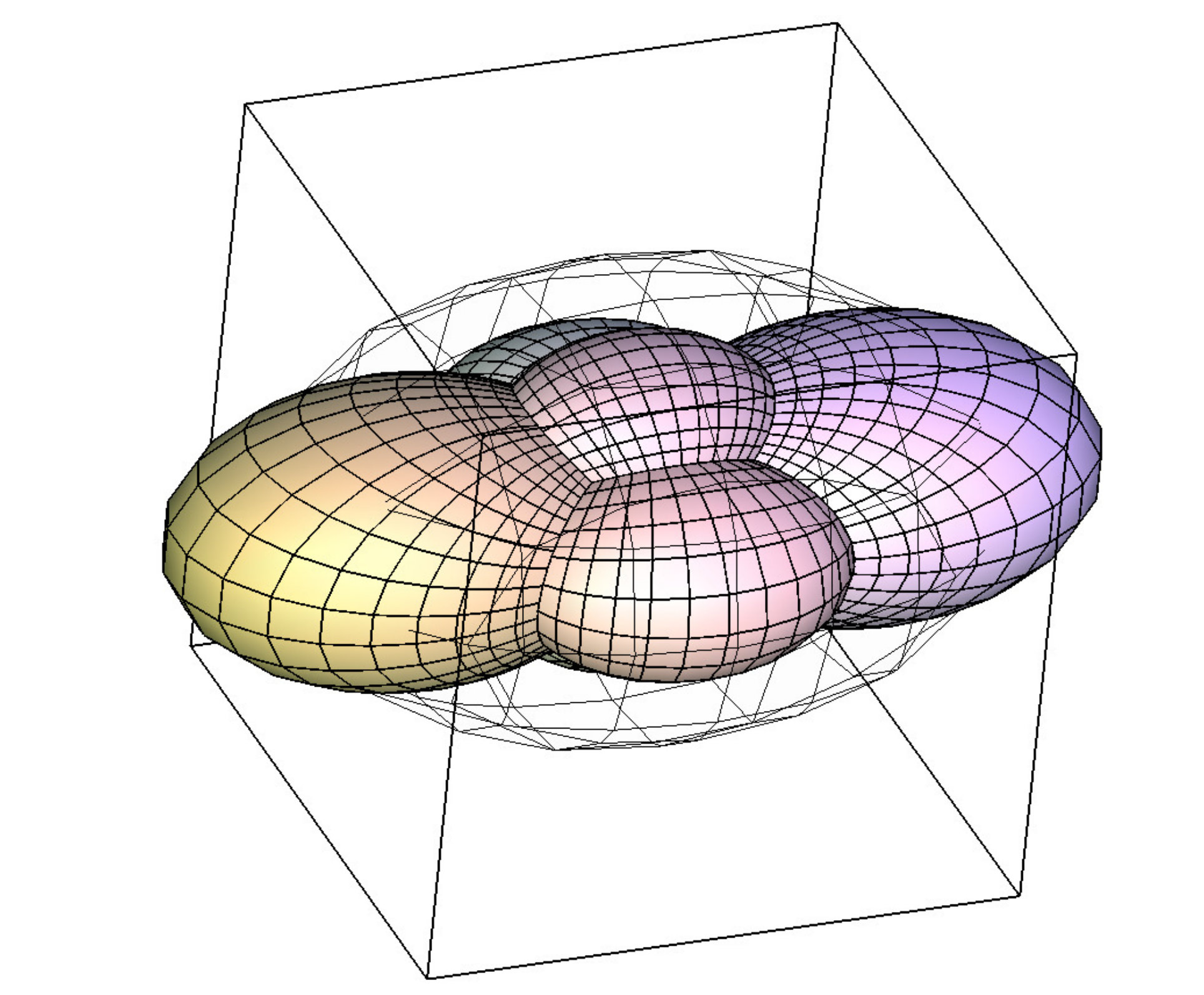}
					\put(55,62){$\mathfrak{a}$}
				\end{overpic}
        \end{minipage}%
        \hfill
        \begin{minipage}[b]{0.50\textwidth}
                \centering
                \begin{overpic}[width=\textwidth]{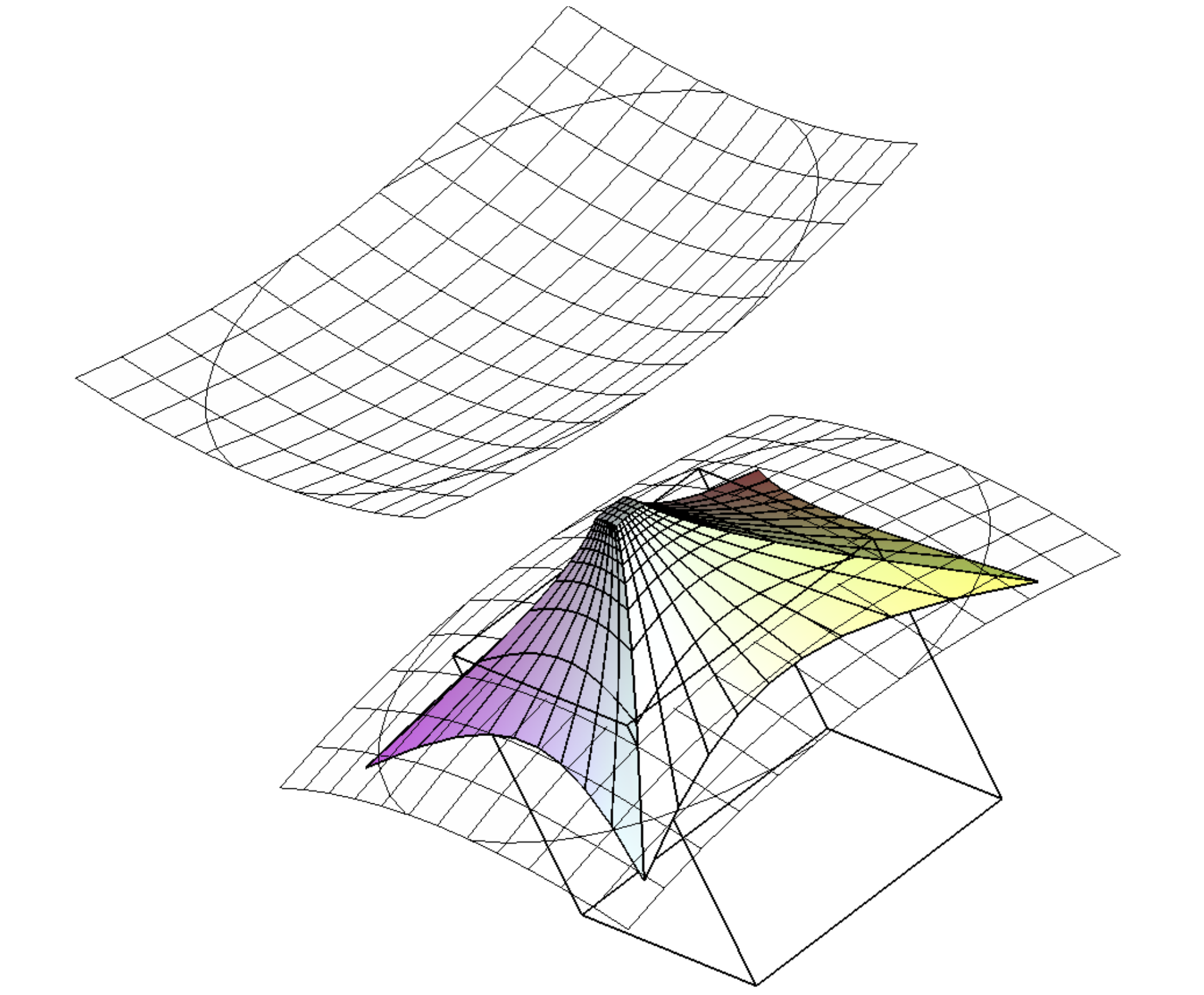}
					\put(12,55){$\mathfrak{a}$}
				\end{overpic}
        \end{minipage}
        \end{center}
        \caption{Inversion of the unit cube with respect to an ellipsoid (left), inversion with respect to a hyperboloid (right)}
        \label{FIG2}
\end{figure}

One main advantage of this method is that we can calculate the image ellipsoid of one face of the cube (one plane) directly by applying the sandwich operator to the plane that is expressed as quadric. For example the plane passing through $P_1=(1,1,1)^\mathrm{T},\,P_2=(1,1,-1)^\mathrm{T},\,P_3=(1,-1,1)^\mathrm{T}$ can be expressed as vector by
\begin{figure}[t]
        \begin{center}
        \begin{minipage}[b]{0.48\textwidth}
                \centering
                \begin{overpic}[width=\textwidth]{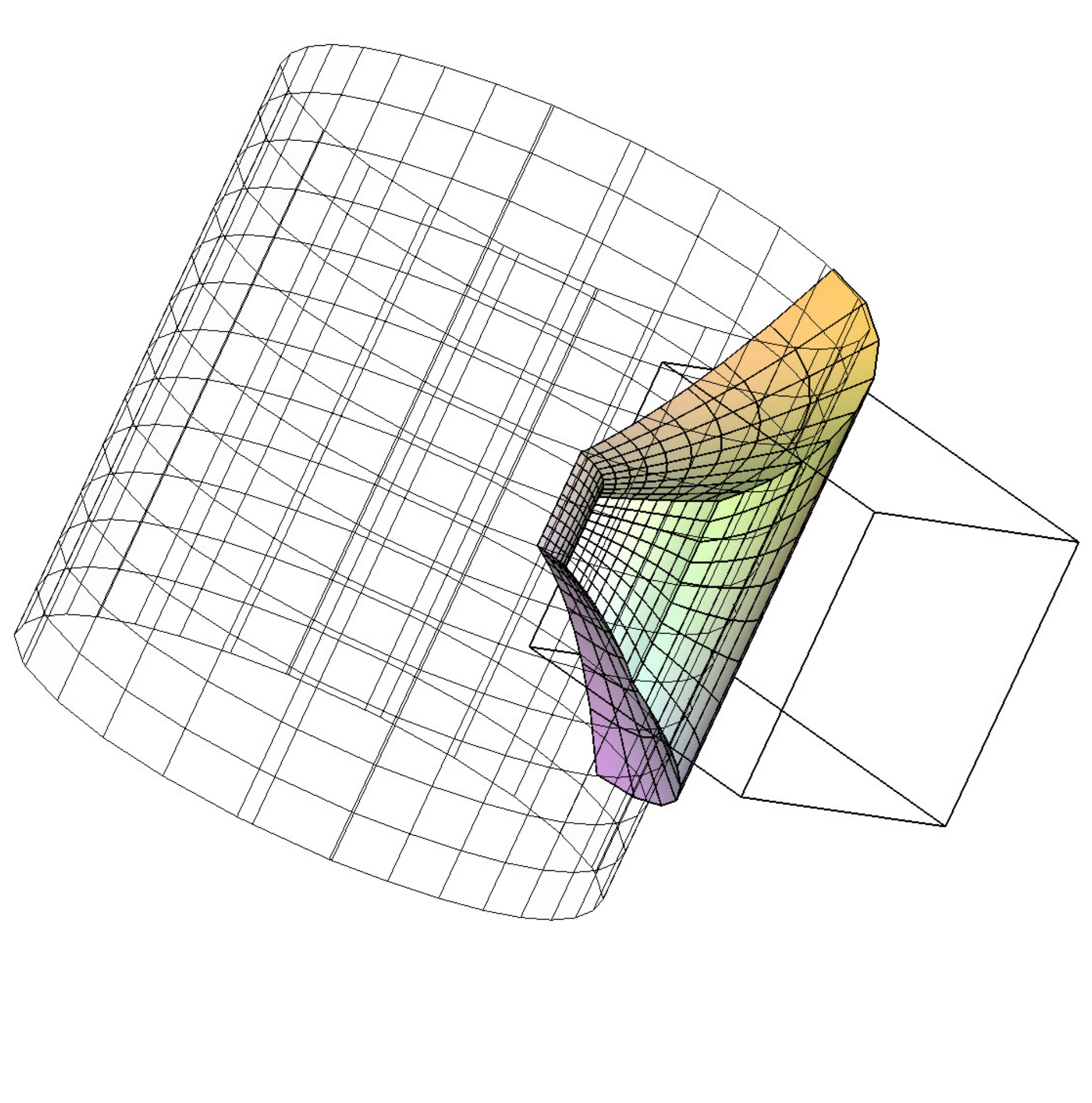}
					\put(10,67){$\mathfrak{a}$}
				\end{overpic}
        \end{minipage}%
        \hfill
        ~ 
        \begin{minipage}[b]{0.48\textwidth}
                \centering
                \begin{overpic}[width=\textwidth]{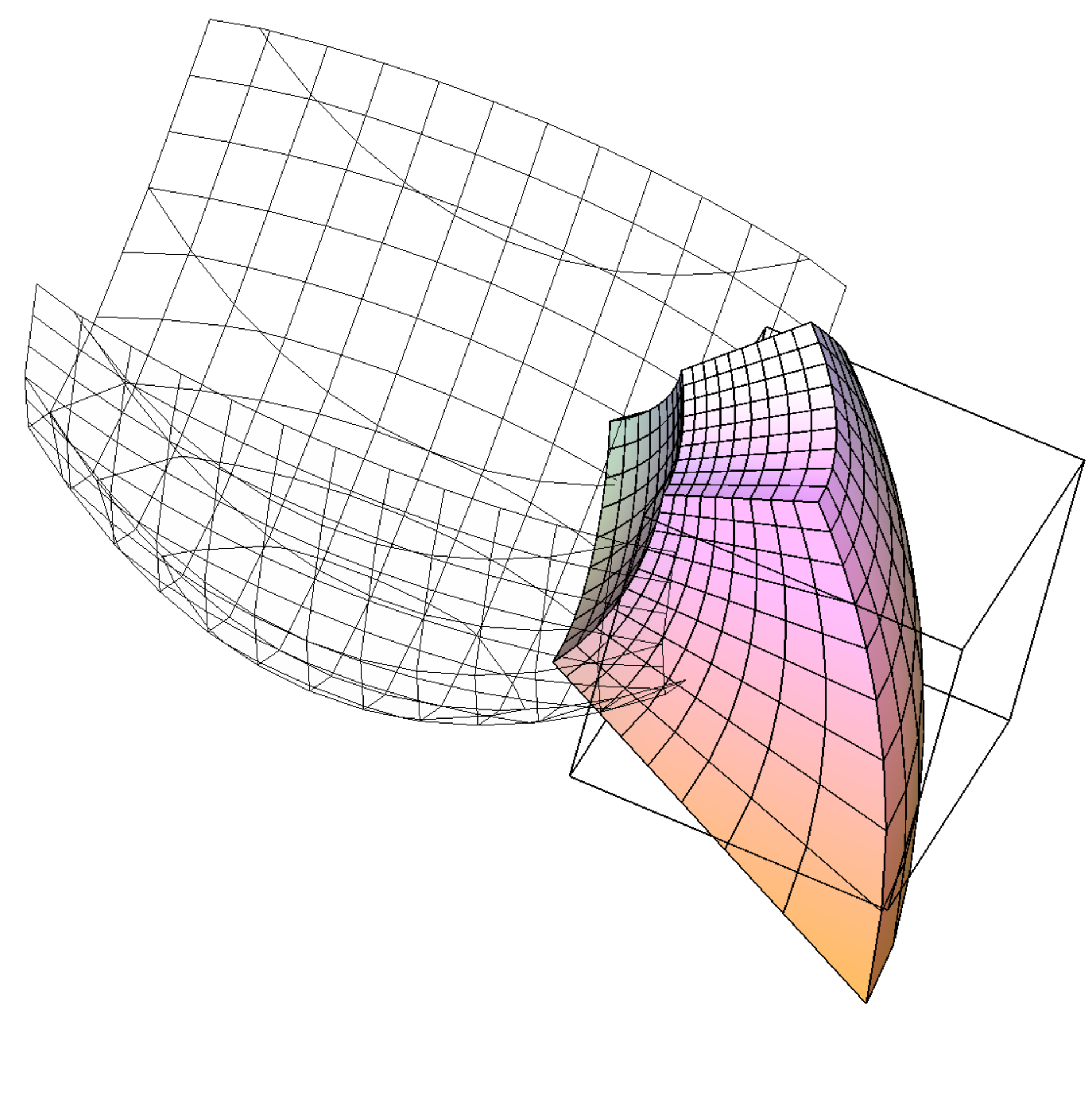}
					\put(6,73){$\mathfrak{a}$}
				\end{overpic}
        \end{minipage}
        \end{center}
        \caption{Inversion with respect to a cylinder (left), inversion with respect to an elliptic paraboloid (right)}
        \label{FIG3}
\end{figure}
\begin{equation*}
\mathfrak{p}=(\eta P_1\wedge\eta P_2\wedge\eta P_3\wedge e_3\wedge e_6 \wedge e_9)\cdot \mathfrak{I}=3e_2+e_3+e_6+e_9.
\end{equation*}
To check that that this indeed the representation of the plane we compute the GIPNS
\[\mathds{NI}_G(\mathfrak{p})=\left\lbrace (x,y,z)^\mathrm{T}\in\mathds{R}^3\mid 1-x=0 \right\rbrace.\]
Applying the sandwich operator to $\mathfrak{p}$ results in
\begin{equation*}
\mathfrak{c}=\alpha(\mathfrak{a})\mathfrak{p} \mathfrak{a}^{-1}=-\frac{200}{27}e_1-3e_2-\frac{32}{3}e_4-\frac{96}{25}e_7
\end{equation*}
with GIPNS
\[\mathds{NI}_G(\mathfrak{c})=\left\lbrace (x,y,z)^\mathrm{T} \in\mathds{R}^3\mid \frac{10^2}{9^2}x^2-x + \frac{4^2}{3^2}y^2+\frac{4^2}{5^2}z^2=0 \right\rbrace. \]
This is one of the ellipsoids displayed in Fig. \ref{FIG2} (left). Furthermore, we can intersect two inner product planes $\mathfrak{p}_1,\mathfrak{p}_2$ to get an inner product line that is an edge of the cube. After that we can apply the sandwich operator to the line and get the intersection curve (an ellipse) of the two ellipsoids that are the images of $\mathfrak{p}_1,\mathfrak{p}_2$.
\FloatBarrier
\begin{remark}
It is more convenient to compute the sandwich operator by a conjugation instead inversion. This means we use $\alpha(\mathfrak{a}) \mathfrak{X} \mathfrak{a}^\ast$ with $\mathfrak{a}\in\clifford^\times_{(p,q,r)}$ and $\mathfrak{X}\in{\bigwedge}^kV$. In general, the modified sandwich operator is easier to handle, because computing the inverse of a Clifford algebra element is extremely expensive. Moreover, we are working in a projective setting, and therefore, multiplication with a homogeneous factor does not change the occurring geometric inner product and outer product null spaces.
\end{remark}
The second example in three-dimensional space is a hyperboloid of two sheets in principal position that is generated by
\[\mathfrak{a}=(\eta P_1\wedge\eta P_2\wedge\eta P_3\wedge\eta P_4\wedge\eta P_5\wedge\eta P_6)\cdot \mathfrak{I}=-6e_1+e_3+\frac{9}{8}e_4+e_6+\frac{9}{8}e_7+e_9.\]
Here, we have $P_1\!=\!(-1,0,0)^\mathrm{T},\,P_2\!=\!(1,0,0)^\mathrm{T},\,P_3\!=\!(2,0,4)^\mathrm{T},\,P_4\!=\!(2,0,-4)^\mathrm{T},$ $P_5\!=\!(2,4,0)^\mathrm{T}$, $P_6\!=\!(2,-4,0)^\mathrm{T}$. 
We calculate the GIPNS 
\[\mathds{NI}_G(\mathfrak{a})=\left\lbrace (x,y,z)^\mathrm{T}\in\mathds{R}^3\mid x^2-1-\frac{3}{16}y^2-\frac{3}{16}z^2=0 \right\rbrace. \]
The mapping applied to pairs of points results in
\[f(x,y,z)=\frac{16}{16x^2-3y^2-3z^2}(x,y,z)^\mathrm{T}.\]
The image of the cube $[1,3]\times[-1,1]\times[-1,1]$ under this mapping is shown in Fig.\ \ref{FIG2} (right).
Fig.\ \ref{FIG3} (left) shows the image of an inversion with respect to a cylinder given by $\mathfrak{a}=3e_1-2e_3+3e_4-2e_6-2e_9$ applied to a cube. The equation of the cylinder is derived as $x^2+y^2=4$. Planes that are not parallel to the axis of the cylinder are mapped to paraboloids. Another example is presented in Fig.\ \ref{FIG3} (right). The inversion quadric is an elliptic paraboloid given by $\mathfrak{a}=3e_1-2e_3+12e_4-2e_6+6e_8-2e_9$ respectively by $x^2+y^2-4z+4=0$.
\FloatBarrier

\section{Conclusion}
\noindent The geometric algebra presented in this article serves for a lot of applications. A generalization of inversions with respect to conics, quadrics and even hyperquadrics in any dimension is possible with the use of the sandwich operator. Hyperquadrics in principal position are simply represented as grade-1 elements. Furthermore, this model serves as a generalization of the conformal geometric algebra, see \cite{dorst:geometricalgebra}. Classical representations of groups are embedded in this algebra naturally.

\section*{acknowledgement}
This work was supported by the research project "Line Geometry for Lightweight Structures", funded by the DFG (German Research Foundation) as part of the SPP 1542.




\begin{thebibliography}{99}
%
%
   \bibitem{dorst:geometricalgebra}
   {\sc Dorst, L; Fontijne, D. and Mann, S.}: {\em Geometric Algebra for Computer Science}.
    Morgan Kaufmann Publishers,\ 2007.
   \bibitem{dorst:gaviewer}
   {\sc Dorst, L; Fontijne, D.}: {\em 3D Euclidean Geometry Through Conformal Geometric Algebra (a GAViewer tutorial)}.
    2005, tutorial, available at
    \url{http://www.science.uva.nl/ga/files/CGAtutorial_v1.3.pdf}
   \bibitem{fontijne:EfficientImplementationofGeometricAlgebra}
   {\sc Fontijne, D.}: {\em Clifford Algebras and the Classical Groups}.
   Morgan Kaufmann Publishers,\ 2007.
   \bibitem{gallier:cliffordalgebrascliffordgroups}
   {\sc Gallier, J.}: {\em Clifford Algebras, Clifford Groups, and a Generalization of the Quaternions: The Pin and Spin Groups}.
    2011, unpublished, available at \url{http://www.cis.upenn.edu/~cis610/clifford.pdf}.
 \bibitem{hestenes:cliffordalgebra}
   {\sc Hestenes D. and Sobczyk, G.}: {\em Clifford Algebra to Geometric Calculus}.
   Kluwer Academic Publishers,\ 1984.    
   \bibitem{perwass:geometricalgebra}
   {\sc Perwass, C.}: {\em Geometric Algebra with Applications in Engineering}.
   Springer-Verlag Berlin Heidelberg,\ 2009.
    \bibitem{porteous:cliffordalgebrasandtheclassicalgroups}
   {\sc Porteous, Ian R.}: {\em Clifford Algebras and the Classical Groups}.
   Cambridge University Press,\ 1995.
   \bibitem{selig:geometricfundamentalsofrobotics}
   {\sc Selig, J.M.}: {\em Geometric Fundamentals of Robotics}.
   2nd ed., Springer,\ 2005.
   \bibitem{zamora:geometricalgebra}
   {\sc Zamora-Esquivel, J.}: {\em $G_{6,3}$ Geometric Algebra; Description and Implementation}.
   Advances in Applied Clifford Algebras 24 (2), 493-514,\ 2014.
\end{thebibliography}


\end{JGGarticle}
\end{document}